\documentclass[reqno,centertags,11pt]{amsart}
\usepackage{verbatim}
\usepackage{tikz, enumerate, array, amssymb, amsmath}
\usepackage{graphicx, epsfig }
\usepackage{amsthm}
\usepackage{amsbsy}
\usepackage{amsfonts}

\usepackage[colorlinks]{hyperref}
\usepackage{mathtools}

\newcommand{\Hmm}[1]{\leavevmode{\marginpar{\tiny%
$\hbox to 0mm{\hspace*{-0.5mm}$\leftarrow$\hss}%
\vcenter{\vrule depth 0.1mm height 0.1mm width \the\marginparwidth}%
\hbox to 0mm{\hss$\rightarrow$\hspace*{-0.5mm}}$\\\relax\raggedright #1}}}

\newcommand{\N}{{\mathbb{N}}}

\newcommand{\R}{{\mathbb{R}}}

\newcommand{\C}{{\mathbb{C}}}

\newcommand{\Z}{{\mathbb{Z}}}

\newcommand{\f}{\frac}

\newcommand{\beq}{\begin{equation}}
\newcommand{\eeq}{\end{equation}}
\newcommand{\bdm}{\begin{displaymath}}
\newcommand{\edm}{\end{displaymath}}
\newcommand{\ba}{\begin{align}}
\newcommand{\ea}{\end{align}}
\newcommand{\bpf}{\begin{proof}}
\newcommand{\epf}{\end{proof}}

\newcommand{\e}{\mathrm{e}}

\newcommand{\veps}{\varepsilon}

\newcommand{\im}{\mathrm{Im}}

\newcommand{\dav}{{d_{\mathrm{av}}}}

\allowdisplaybreaks

\newcommand{\calC}{\mathcal{C}}

\newtheorem{theorem}{Theorem}
\newtheorem{proposition}[theorem]{Proposition}
\newtheorem{lemma}[theorem]{Lemma}

\theoremstyle{definition}

\newtheorem{definition}[theorem]{Definition}

\newtheorem{remark}[theorem]{Remark}
\newtheorem{remarks}[theorem]{Remarks}

\newcounter{theoremi}[theorem]

\numberwithin{theorem}{section}
\numberwithin{equation}{section}

\newcounter{smalllist}

\newcounter{listi}
\newenvironment{theoremlist}{\begin{list}{{\rm(\roman{listi})}}{%
\setlength{\topsep}{0mm}\setlength{\parsep}{0mm}\setlength{\itemsep}{0mm}%
\setlength{\labelwidth}{1.5em}\setlength{\leftmargin}{1.7em}\usecounter{listi}%
}}{\end{list}}

\newcounter{smallenum}

\newcounter{assumptions}

\begin{document}
\title[Ground states for 2D DMNLS]{Ground states of the two-dimensional dispersion managed nonlinear Schr\"{o}dinger equation}
\author{Mi-Ran Choi, Young-Ran Lee}




\address{Department of Mathematics, Sogang University, 35 Baekbeom-ro,
    Mapo-gu, Seoul 04107, South Korea.}%
\email{mrchoi@sogang.ac.kr, younglee@sogang.ac.kr}

\begin{abstract}
We consider the variational problem with a mass constraint arising from the two-dimensional dispersion managed nonlinear Schr\"odinger equation with power-law type nonlinearity.
We prove a threshold phenomenon with respect to mass for the existence of minimizers for all possible powers of nonlinearities, including at the threshold itself. This threshold is closely related to the best constant for the Gagliardo-Nirenberg-Strichartz type inequality whose extremizers are found as a byproduct.
\end{abstract}

\date{\today}
\maketitle

\section{\bf Introduction}

We consider the  constrained variational problem
\beq \label{eq:min}
E_\lambda:= \inf\{ H(f) : f\in H^1(\R^d),\, \|f\|_{L^2}^2=\lambda\},
\eeq
where the associated Hamiltonian is given  by
\beq\label{eq:Hamiltonian}
H(f):=\f{\dav}{2}\|\nabla f\|_{L^2}^2 - \f{1}{p+1}\int_0^1\int_{\R^d} |e^{ir\Delta} f(x)|^{p+1}dxdr,
\eeq
where $\dav>0$, $\lambda>0$, $d=1,2$, $p>1$, and $e^{ir\Delta}$ is the free Schr\"odinger evolution.
The  problem \eqref{eq:min} is strongly related to the search for breather-type solutions to the dispersion managed nonlinear Schr\"odinger equation (DMNLS)
\beq\label{NLS}
i\partial_t u+d(t)\Delta u +|u|^{p-1}u=0,\quad u:\R\times \R^d\to \C.
\eeq
Here, the dispersion $d(t)$ is given by  $d(t)=\veps^{-1}d_0(t/\veps)+\dav$ with small parameter $\veps>0$, where $\dav$ denotes the average component of dispersion and $d_0$ its $2$-periodic mean zero part, which is assumed to be piecewise constant, given by  $d_0=1_{[0,1)}-1_{[1,2)}$.

The equation \eqref{NLS} appears in a variety of physical contexts, including the propagation of electromagnetic pulses in optical fiber communication, the beam propagation in waveguide arrays, and the investigation of nonlinear matter waves in Bose-Einstein condensates, see \cite{Abdullaev,Chraplyvy, PEREZGARCIA200631, Serkin, SS}.
From both a mathematical and physical viewpoint, an important issue for such an NLS equation is finding stable standing wave solutions, known as solitons, which arise from balancing  the competing effects of nonlinearity and dispersion.  

Performing an appropriate transformation in \eqref{NLS} and then averaging over one
 period yields the averaged equation in the form
 \beq\label{eq:DMNLS}
i\partial_t u +\dav \Delta u +\int_0^1 e^{-ir\Delta}(|e^{ir\Delta}u|^{p-1}e^{ir\Delta}u)dr=0,
\eeq
which was first derived by Gabitov and Turitsyn in \cite{Gabitov96, Gabitov96b} for one-dimensional cubic DMNLS.
 The validity of this averaging process is confirmed in \cite{ CKL, CLA, ZGJT01}.
Standing wave solutions of \eqref{eq:DMNLS} of the form $u(t,x)=e^{i\omega t} f(x)$ with $\omega\in \R$ are solutions of the nonlinear eigenvalue equation
\beq \label{eq:EL}
-\omega f=-\dav \Delta f -  \int_0^1 e^{-ir\Delta}(|e^{ir\Delta}f|^{p-1}e^{ir\Delta}f)dr.
\eeq
By standard methods in the calculus of variations, equation \eqref{eq:EL} is the Euler-Lagrange equation for the Hamiltonian \eqref{eq:Hamiltonian}.

\smallskip

The minimization problem \eqref{eq:min} for the one-dimensional case has been well studied, whose minimizer is known as a dispersion managed soliton in optical communications. It was shown that when $\dav> 0$ and $1<p<5$, this problem possesses a minimizer in $H^1(\R)$ for all $\lambda>0$, see \cite{CHLT, ZGJT01}.
However, when $5\leq p<9$, there exists a threshold $0<\lambda_{cr}<\infty$ such that if $0<\lambda<\lambda_{cr}$, then $E_\lambda=0$  and it has no minimizer; if $\lambda>\lambda_{cr}$, then $E_\lambda<0$ and it attains a minimizer in $H^1(\R)$, see \cite{CHLT, HLRZ}.
The minimization problem in the case of zero-average dispersion is more subtle due to the loss of compactness of a minimizing sequence.
Nevertheless, the one-dimensional minimization problem has been successfully analyzed, see \cite{CHLT, HL, Kunze, Stanislavova}. Indeed, if $1<p < 5$, then there is a minimizer in $L^2(\R)$, while no minimizer exists for $p=5$. For the cubic nonlinearity, the smoothness and decay of the minimizers have been studied in \cite{EHL, HL2009, Stanislavova} for the zero average dispersion, and in \cite{GH} for the positive average dispersion.

A natural extension of this problem is to study the two-dimensional case which is interested in both mathematics and physics. In \cite{ZGJT01}, the authors discussed the existence of a minimizer for the cubic nonlinearity when $\dav>0$. Specifically, they used the Gaussian functions to achieve $E_\lambda<0$ for sufficiently large $\lambda > 0$ only and then applied Lions' concentration compactness principle to obtain a minimizer. We extend this to the power-law type nonlinearities and establish the existence of a threshold, $\lambda_{cr}$, which is our first main result.

\begin{theorem}\label{thm:existence}
Assume $d=2$ and $1< p < 5$. Then
\[
-\infty < E_\lambda \leq 0
\]
for every $\lambda>0$, that is, the constrained minimization problem is well-defined.
Moreover, 
the following hold:
\begin{theoremlist}
    \item There exists a threshold $0\leq \lambda_{cr}<\infty$ such that
    \[
    E_\lambda=0 \quad \mbox{if}\;\; 0<\lambda \leq  \lambda_{cr}, \quad -\infty< E_\lambda<0\quad \mbox{if}\;\; \lambda >\lambda_{cr}.
    \]
     \item If $0<\lambda < \lambda_{cr}$, then there is no minimizer for \eqref{eq:min}.
 \item If $\lambda \geq \lambda_{cr}$ and $\lambda>0$, then there exists a minimizer for \eqref{eq:min}  which solves
  \eqref{eq:EL} for some positive Lagrange multiplier $\omega>-2E_{\lambda}/\lambda$.
 \item If $1< p <3$, then $ \lambda_{cr}=0$.
 \item If $3 \le p <5$, then $\lambda_{cr}>0$. In fact,
 \[
 \lambda_{cr}=\left( \frac{\dav(p+1)}{2\calC_p}\right)^{\frac{2}{p-1}},
 \]
 where $\calC_p$ is the best constant for the Gagliardo-Nirenberg-Strichartz type inequality
 \beq\label{ineq:lambda_cr}
   \int_0 ^1 \|e^{ir\Delta}f\|_{L^p+1}^{p+1} dr\leq \calC_p \|\nabla f\|_{L^2}^{2}\,\|f\|_{L^2}^{p-1}
 \eeq
 for all $f\in H^1(\R^2)$.
 \end{theoremlist}
\end{theorem}

\begin{remark}
\begin{theoremlist}
\item  Such a threshold for the existence of minimizers in dimension one was introduced in \cite{CHLT, HLRZ}, where the authors considered a large class of nonlinearities.
However, the existence/nonexistence of minimizers for the threshold itself was remained open, see \cite[Theorems 1.2]{CHLT}.
By the same argument in the proof of Theorem \ref{thm:existence}, it can be proved that $E_{\lambda_{cr}}=0$ but there exists a minimizer for $E_{\lambda_{cr}}$. Moreover,
\[
\lambda_{cr}=\begin{cases}
               0 & \mbox{if } 1<p<5 \\
               \left( \frac{\dav(p+1)}{2\calC_p}\right)^{\frac{2}{p-1}} & \mbox{if } 5 \le p<9 .
             \end{cases}
\]
Here, $\calC_p$ is the best constant for the one-dimensional Gagliardo-Nirenberg-Strichartz type inequality
 \[
 \int_0^1\|e^{ir\partial_x^2}f\|_{L^{p+1}(\R)}^{p+1} dr \leq \calC_p\|\partial_x f\|_{L^2(\R)}^2\|f\|_{L^2(\R)}^{p-1}
 \]
 for all $f\in H^1(\R)$ whose proof can be found in \cite[Lemma 2.3]{CHLW}.
\item The inequality \eqref{ineq:lambda_cr} holds true, see Lemma \ref{lem:Strichartz type 2}.
The definition of $E_\lambda$ and Theorem \ref{thm:existence} (iii), (v) provide the existence of an extremizer for \eqref{ineq:lambda_cr} for $3 \le p <5$. Indeed, since a minimizer $f$ exists for $E_{\lambda_{cr}}$, by \eqref{ineq:lambda_cr},
    \[
    \begin{aligned}
    0=E_{\lambda_{cr}}=H(f) &=\frac{\dav}{2}\|\nabla f\|_{L^2}^2 -\frac{1}{p+1} \int_0^1 \|e^{ir\Delta}f \|_{L^{p+1}}^{p+1}dr\\
&\geq       \frac{\dav}{2}\|\nabla f\|_{L^2}^2 -\frac{\calC_p}{p+1} \|\nabla f \|_{L^{2}}^{2} \|f\|_{L^2}^{p-1}\\
& =\|\nabla f\|_{L^2}^2 \left( \frac{\dav}{2}-\frac{\calC_p}{p+1}\lambda_{cr}^{\frac{p-1}{2}} \right)=0.
      \end{aligned}
  \]
Thus,
$f$ is an extremizer for \eqref{ineq:lambda_cr}.

\item When $\dav=0$ and $1<p\leq3$, the minimization problem is well-defined due to the Strichartz estimate, see Lemma \ref{lem:Strichartz type 1}. However, it is shown in \cite{Stanislavova} that for $p=3$, there is no minimizer in $L^2(\R^2)$. For other values $p$, the existence of minimizers remains an open problem of great interest in both mathematics and physics.
We also note that when $\dav<0$, the minimization problem is not well-defined, i.e., $E_\lambda=-\infty$ for all $\lambda>0$, and thus cannot have a global minimizer.
\end{theoremlist}
\end{remark}

\bigskip

Concerning the question whether the range of exponent in Theorem \ref{thm:existence} is optimal,
the following result provides an interesting answer for the critical exponent $p=5$.

\begin{theorem}\label{thm:critical}
        If $d=2$ and $p=5$, then there exists $\lambda_{cr}>0$ such that
        \[
        E_\lambda= \begin{cases}
0 \quad &\mbox{for all}\quad 0<\lambda \leq\lambda_{cr} \\[.5ex]
      -\infty \quad &\mbox{for all}\quad \lambda >\lambda_{cr}.
            \end{cases}
\]
        When $0< \lambda \leq\lambda_{cr}$, $E_{\lambda}=0$ is not achieved.  Moreover, $\lambda_{cr}^2=3\dav/\calC_5$, where $\calC_5$ is the best constant for the Gagliardo-Nirenberg-Strichartz inequality
\beq\label{eq:G-Ntype-0}
\int_0^1\|e^{ir\Delta}f \|_{L^{6}}^{6}dr \leq \calC_5\|\nabla f\|_{L^2}^{2}\|f\|_{L^2}^{4}.
\eeq

\end{theorem}

\begin{remark}
    \begin{theoremlist}
        \item
      For $p>5$, $E_\lambda=-\infty$ as noted in Remark \ref{rem:unbounded}.
    \item The inequality \eqref{eq:G-Ntype-0} holds true, see Lemma \ref{lem:Strichartz type 2}. However, there is no extremizer for \eqref{eq:G-Ntype-0} as shown in Proposition \ref{prop:no maximizer}.
    \item In the one-dimensional mass-critical case $p=9$, we have the analogous result to Theorem \ref{thm:critical}, more precisely, $E_{\lambda}=0$ for all $0<\lambda \leq\lambda_{cr} $ and $E_{\lambda}=-\infty$ for all $\lambda >\lambda_{cr}$, where   $\lambda_{cr}^4=5\dav /\calC_9$. Here, $\calC_9$ is the best constant for the one-dimensional Gagliardo-Nirenberg-Strichartz inequality
\beq\label{eq:G-Ntype-1d}
\int_0^1\|e^{ir\partial_x^2}f \|_{L^{10}(\R)}^{10}dr \leq \calC_9 \|\partial_xf\|_{L^2(\R)}^{2}\|f\|_{L^2(\R)}^{8}.
\eeq
 Moreover, $E_{\lambda}=0$ is not achieved if $0<\lambda \leq\lambda_{cr}$.
 We can prove this by employing a similar argument as in the proof of Theorem \ref{thm:critical} with norm quantities of the critical element for \eqref{eq:G-Ntype-1d}, which can be found in \cite{choi2023global}.
    \end{theoremlist}
\end{remark}

\bigskip

The paper is organized as follows: In Section \ref{sec: Preliminary}, we give some preliminary
results.
In Section \ref{sec:minimizer}, we first establish the properties and existence of minimizers for $E_\lambda$ and then prove Theorem \ref{thm:existence}.
Finally, we present the proof of Theorem \ref{thm:critical} in Section \ref{sec:Proof of Theorem 1.3}.

\section{Preliminaries} \label{sec: Preliminary}

 Here and below, we use the convention $f\lesssim g$ if there exists a finite constant $C>0$ such that  $f\leq C g$.
We say that the pair $(q,r)$ is admissible (in two dimensions) if
\[
\frac{2}{q}+\frac{2}{r}=1 \quad \mbox{with} \quad 2\leq r<  \infty.
\]

The proof of our main results relies heavily on the inequalities in the following three lemmas which are based on the two-dimensional Strichartz  inequality
\[
\left( \int_\R \|e^{it\Delta } f\|_{L^r}^qdt\right)^{\frac{1}{q}}\lesssim \|f\|_{L^2}
\]
where $(q,r)$ is an admissible pair. The sharp constant in the two-dimensional Strichartz inequality for the admissible pair $(4,4)$ is known to be $2^{-1/2}$, see \cite{Foschi2007, 8180258}. They, also, showed that the equality holds if and only if $f$ is a Gaussian.

\begin{lemma}\label{lem:Strichartz type 1}
 If $2\leq q\leq 4$, then
    \beq\label{ineq:Strichartz type 1}
    \int_0 ^1 \|e^{ir\Delta}f\|_{L^q}^q dr \lesssim \|f\|_{L^2}^{q}
    \eeq
for all $f\in L^2(\R^2)$.
\end{lemma}
\begin{proof}
  We  use $q=(4-q)+(2q-4)$ and the H\"older inequality to get
    \[
    \int_0^1 \!\int_{\R^2}|e^{ir\Delta}f(x)|^q dx\,dr \leq
   \left( \int_0^1 \!\int_{\R^2}|e^{ir\Delta}f(x)|^2 dx\,dr\right)^{\f{4-q}{2}}\left( \int_0^1 \!\int_{\R^2}|e^{ir\Delta}f(x)|^4 dx\,dr\right)^{\f{q-2}{2}}.
    \]
Then \eqref{ineq:Strichartz type 1} follows from the fact that $e^{ir\Delta}$ is unitary on $L^2(\R^2)$ and the Strichartz inequality.
\end{proof}

\begin{lemma}\label{lem:Strichartz type 2}
\begin{theoremlist}
    \item
    If $2\leq q \leq 6$, then
     \beq\label{ineq:Strichartz type 2}
       \int_0 ^1 \|e^{ir\Delta}f\|_{L^q}^q dr\lesssim \|\nabla f\|_{L^2}^{\f{q-2}{2}} \,\|f\|_{L^2}^{\f{q+2}{2}}
     \eeq
     for all $f\in H^1(\R^2)$.
     \item
    If $4 \leq q\leq 6$, then
     \beq\label{ineq:Strichartz type 4}
  \int_0 ^1 \|e^{ir\Delta}f\|_{L^q}^q dr\lesssim \|\nabla f\|_{L^2}^{2}\,\|f\|_{L^2}^{q-2}
    \eeq
     for all $f\in H^1(\R^2)$.
\end{theoremlist}
\end{lemma}
\begin{proof}
It is trivial that \eqref{ineq:Strichartz type 2} for $q=2$ and \eqref{ineq:Strichartz type 4} for $q=4$ hold by the unitarity of $e^{ir\Delta}$ and  the Gagliado-Nirenberg inequality
\[
\|g\|_{L^4}^4\lesssim \|\nabla g\|_{L^2}^2 \|g\|_{L^2}^2.
\]
Now, to prove the others, we observe that
\beq\label{ineq:L_infty}
\begin{aligned}
\int_0^1 \!\int_{\R^2}|e^{ir\Delta}f(x)|^q dx\,dr
& \leq
  \int_0^1 \|e^{ir\Delta}f\|_{L^\infty}^{q-2} \int_{\R^2}|e^{ir\Delta}f(x)|^2 dx\,dr\\
 & = \|f\|_{L^2}^2\int_0^1\|e^{ir\Delta}f\|_{L^\infty}^{q-2} dr.
  \end{aligned}
\eeq
The remaining proof is based on the following Gagliardo-Nirenberg inequality
\beq\label{ineq:well-known}
\|g\|_{L^\infty}\lesssim \|\nabla g\|_{L^\gamma}^{2/\gamma} \|g\|_{L^\gamma}^{1-2/\gamma}
\eeq
for $\gamma >2$.

First, if $2< q\leq 6$, then using \eqref{ineq:well-known} with $\gamma=4$ and then the H\"older inequality with $\f{8}{q-2}$, $\f{8}{q-2}$, and $\f{4}{6-q}$, we have
\[
\begin{aligned}
\int_0^1\|e^{ir\Delta}f\|_{L^\infty}^{q-2} dr
&\lesssim \int_0^1 \|e^{ir\Delta}\nabla f\|_{L^4}^{(q-2)/2} \cdot \|e^{ir\Delta}f\|_{L^4}^{(q-2)/2} \cdot 1 \, dr\\
& \le \left(\int_0^1 \|e^{ir\Delta}\nabla f\|_{L^4}^4\,dr\right)^{\f{q-2}{8}}\left(\int_0^1 \|e^{ir\Delta}f\|_{L^4}^4\,dr\right)^{\f{q-2}{8}}
\end{aligned}
\]
which together with the Strichartz inequality
and \eqref{ineq:L_infty} yields \eqref{ineq:Strichartz type 2}.

Second, for $4<q<6$, we use \eqref{ineq:well-known} with $\gamma=q-2>2$, the H\"older inequality in $r$ with the exponents $\f{q-2}{q-4}$, $\f{2(q-2)}{(q-4)^2}$, and $\f{2}{6-q}$, and  the Strichartz inequality with
the admissible pair $(2(q-2)/(q-4),q-2)$ to get
\[
\begin{aligned}
\int_0^1\|e^{ir\Delta}f\|_{L^\infty}^{q-2} dr
& \lesssim \int_0^1 \|e^{ir\Delta}\nabla f\|_{L^{q-2}}^{2} \, \|e^{ir\Delta}f\|_{L^{q-2}}^{q-4} \cdot 1\, dr\\
& \lesssim \|\nabla f\|_{L^2}^{2}\|f\|^{q-4}_{L^2}.
\end{aligned}
\]
Therefore, by \eqref{ineq:L_infty}, we obtain \eqref{ineq:Strichartz type 4}.
\end{proof}

\begin{remark}
\begin{theoremlist}
  \item The inequality \eqref{ineq:Strichartz type 4} holds for $4\leq q\leq 6$ only. Indeed, using the Gaussian test function $g_{\sigma_0}(x):=e^{-|x|^2/\sigma_0}$ for $\sigma_0>0$, we have
\begin{equation}\label{eq:gaussian}
  \|g_{\sigma_0}\|_{L^2}^2=\frac{\pi}{2}\sigma_0,\quad \|\nabla g_{\sigma_0}\|_{L^2}^2=\pi,\quad
  \int_0^1 \|e^{ir\Delta}g_{\sigma_0}\|_{L^q}^q dr=\frac{\pi}{q}\,\sigma_0^{q-1}\int_{0}^{1}|\sigma(r)|^{2-q} dr
\end{equation}
since
\[
e^{ir\Delta} g_{\sigma_0}(x)= \f{\sigma_0}{\sigma(r)}\, e^{-\f{|x|^2}{\sigma(r)}},
\]
where $\sigma(r)=\sigma_0+4ir$. Therefore, if $2 \le q < 4$, then
\[
\frac{\int_0^1 \|e^{ir\Delta}g_{\sigma_0}\|_{L^q}^q dr}{\|\nabla g_{\sigma_0}\|_{L^2}^2 \|g_{\sigma_0}\|_{L^2}^{q-2}}=\frac{1}{q} \left(\frac{2}{\pi}\right)^{\frac{q-2}{2}}\sigma_0^{\frac{4-q}{2}}\int_0^1 \left(\frac{1}{1+(4r/\sigma_0)^2}\right)^{\frac{q-2}{2}}dr\to \infty
\]
as $\sigma_0\to \infty$. For $q>6$, we take the change of variables to get
\[
\frac{\int_0^1 \|e^{ir\Delta}g_{\sigma_0}\|_{L^q}^q dr}{\|\nabla g_{\sigma_0}\|_{L^2}^2 \|g_{\sigma_0}\|_{L^2}^{q-2}}=\frac{1}{q} \left(\frac{2}{\pi}\right)^{\frac{q-2}{2}}\sigma_0^{\frac{6-q}{2}}\int_0^{1/\sigma_0} \left(\frac{1}{1+(4r)^2}\right)^{\frac{q-2}{2}}dr\to \infty
\]
as $\sigma_0\to 0$.
\item
We note that if $q=6$, then the inequalities \eqref{ineq:Strichartz type 2} and \eqref{ineq:Strichartz type 4} are identical and do not possess an extremizer, see Proposition \ref{prop:no maximizer}.
\end{theoremlist}
\end{remark}

\begin{lemma} \label{lem:Lipshhitz}
If $2\leq q<\infty$, then
\beq\label{ineq:N2}
\begin{aligned}
\Big|\int_0 ^1 \|e^{ir\Delta} f \|_{L^{q}}^{q} -\|e^{ir\Delta} g \|_{L^{q}}^{q}\,dr\Big|
\lesssim  \left(\|f\|_{H^1}^{q-1}+\|g\|_{H^1}^{q-1}  \right)\|f-g\|_{L^2}
\end{aligned}
\eeq
for any $f, g \in H^1(\R^2)$.
\end{lemma}

\begin{proof}
Using the Cauchy-Schwarz inequality in $x$, one sees that
    \[
    \begin{aligned}
      &\Big|\int_0 ^1 \|e^{ir\Delta} f \|_{L^{q}}^{q} -\|e^{ir\Delta} g \|_{L^{q}}^{q}\,dr\Big|    \\
      & \leq \int_0 ^1 \!\int_{\R^2}  \Big| |e^{ir\Delta} f(x)|^{q} -|e^{ir\Delta} g(x)|^{q}\Big|dx\,dr \\
  &  \lesssim \int_0 ^1\! \int_{\R^2} \left(|e^{ir\Delta}f(x)|^{q-1}+|e^{ir\Delta} g(x)|^{q-1}\right)|e^{ir\Delta} (f-g)(x)| dx\,dr \\
  &  \lesssim \|f-g\|_{L^2}  \int_0 ^1\left(\|e^{ir\Delta} f\|_{L^{2(q-1)}}^{q-1} +\|e^{ir\Delta}g\|_{L^{2(q-1)}}^{q-1} \right) dr,
     \end{aligned}
    \]
    where we used the elementary inequality
    \[
    ||z|^q-|w|^q|\lesssim (|z|^{q-1}+|w|^{q-1})|z-w|, \quad q\geq1
    \]
    for any $z, w\in \C$.
To prove \eqref{ineq:N2}, it suffices to show that
\beq\label{ineq:2q-2}
 \int_0 ^1\|e^{ir\Delta} f\|_{L^{2(q-1)}}^{q-1}dr \lesssim \|f\|_{H^1}^{q-1}.
\eeq
For $2\leq q \leq 4$, \eqref{ineq:2q-2} follows from the Cauchy-Schwarz inequality in $r$ and Lemma \ref{lem:Strichartz type 2}.
For $4< q<\infty$, it follows that
\beq\label{eq:local}
 \int_0 ^1\|e^{ir\Delta} f\|_{L^{2(q-1)}}^{q-1}dr\leq \|f\|_{L^2}\int_0^1\|e^{ir\Delta}f\|_{L^\infty}^{q-2} dr.
\eeq
Using \eqref{ineq:well-known} with $\gamma=2(q-2)/(q-4) >2$, the H\"older inequality with $\f{q-2}{q-4}$ and $\f{q-2}{2}$, and  the  Strichartz inequality, 
we have
\[
\begin{aligned}
\int_0^1\|e^{ir\Delta}f\|_{L^\infty}^{q-2} dr
& \lesssim \int_0^1 \|e^{ir\Delta}\nabla f\|_{L^\gamma}^{q-4} \, \|e^{ir\Delta}f\|_{L^\gamma}^2 \, dr\\
&\le \left(\int_0^1 \|e^{ir\Delta}\nabla f\|_{L^\gamma }^{q-2}\,dr\right)^{\f{q-4}{q-2}}\left(\int_0^1 \|e^{ir\Delta}f\|_{L^\gamma}^{q-2}\,dr\right)^{\f{2}{q-2}}\\
& \lesssim \|\nabla f\|_{L^2}^{q-4}\|f\|^2_{L^2}
\end{aligned}
\]
which together with \eqref{eq:local} yields \eqref{ineq:2q-2}.
\end{proof}

\section{Constrained minimization problem} \label{sec:minimizer}
In this section, 
we give a complete proof of the existence/nonexistence of a minimizer for the constrained minimization problem
\beq \label{eq:min_problem}
E_\lambda= \inf\{ H(f) : f\in H^1(\R^2), \; \|f\|_{L^2}^2=\lambda\}
\eeq
where 
the associated Hamiltonian is given  by
\beq\label{eq:energy}
H(f)=\f{\dav}{2}\|\nabla f\|_{L^2}^2 - \f{1}{p+1}\int_0^1\!\int_{\R^2} |e^{ir\Delta} f(x)|^{p+1}dxdr.
\eeq

We first conduct a detailed study on the basic properties of $E_\lambda$.

\begin{lemma}\label{lem:well-defined}
If $1<p<5$, then  for every $\lambda>0$
    \[
    -\infty < E_\lambda \leq 0.
    \]
In particular, if $1<p<3$, then $E_\lambda <0$ for every $\lambda>0$.
\end{lemma}

\begin{proof}
Let $1<p<5$ and $\lambda >0$. If $f\in H^1(\R^2)$ with $\|f\|_{L^2}^2=\lambda$, then by Lemma \ref{lem:Strichartz type 2} (i)
\[
H(f) \geq \f{\dav}{2}\|\nabla f\|_{L^2}^2 - C\lambda^{\f{p+3}{4}}\|\nabla f\|_{L^2}^{\f{p-1}{2}}
\]
and therefore  $E_\lambda>-\infty$ since $1<p<5$.

To prove $E_\lambda \leq 0$, we consider the Gaussian test function $g_{\sigma_0}$, defined by
\beq\label{gaussian}
g_{\sigma_0}(x)=\left( \f{2\lambda }{\pi \sigma_0}\right)^{\f{1}{2}}e^{-\f{|x|^2}{\sigma_0}}\quad \mbox{with}\quad \sigma_0>0,
\eeq
similar to the one-dimensional case as discussed in \cite[Lemma B.3]{CHLT} and \cite[Theorem B.1]{ZGJT01}.
It follows from $\|g_{\sigma_0}\|_{L^2}^2=\lambda$ and $\|\nabla g_{\sigma_0}\|_{L^2}^2= \f{2\lambda}{\sigma_0}$ that
\[
E_\lambda \le H(g_{\sigma_0})\leq \f{\dav}{2}\|\nabla g_{\sigma_0}\|_{L^2}^2 = \f{\dav \lambda }{\sigma_0}.
\]
Thus, by letting $\sigma_0\to \infty$, we have $E_\lambda\leq 0$.

It remains to show that $E_\lambda <0$ when $1<p<3$.
It follows from \eqref{eq:gaussian} that
\[
\int_0^1\|e^{ir\Delta} g_{\sigma_0}\|_{L^{p+1}}^{p+1}\,dr
= \left( \f{2\lambda }{\pi }\right)^{\f{p+1}{2}} \left (\f{\pi}{p+1}\right)\sigma_0^{\f{1-p}{2}}  \int_0^1 \left( \f{1}{1+(4r/\sigma_0)^2} \right)^{\f{p-1}{2}}dr.
\]
Thus, we have
\beq\label{eq:gaussian_H}
\begin{aligned}
H(g_{\sigma_0})&= \f{\dav \lambda}{\sigma_0}\left(1- \f{(2\lambda)^{\f{p+1}{2}}\pi^{\f{1-p}{2}} }{\dav \lambda (p+1)^2} \, \sigma_0^{\f{3-p}{2}}  \int_0^1 \left( \f{1}{1+(4r/\sigma_0)^2} \right)^{\f{p-1}{2}}dr\right) <0
\end{aligned}
\eeq
for sufficiently large $\sigma_0$ since the integral converges to one as $\sigma_0 \to \infty$ and $1<p<3$.
\end{proof}

Next, we give the following strict sub-additivity of $E_\lambda$.

\begin{proposition}[Strict sub-additivity] \label{prop:subadditivity}
 Assume $1<p<5$. For any $\lambda_1, \lambda_2>0$,
\beq\label{ineq:additivity}
E_{\lambda_1}+E_{\lambda_2}\geq  E_{\lambda_1+\lambda_2}.
\eeq
If, in addition, $E_{\lambda_1+\lambda_2}<0$, then the energy is strictly sub-additive, that is,
    \beq\label{Strict Sub-additivity}
E_{\lambda_1}+E_{\lambda_2}> E_{\lambda_1+\lambda_2}.
\eeq
\end{proposition}
\begin{proof}
    First, note that
    \beq \label{ineq:additivity1}
    E_{\mu \lambda}\geq \mu^{\frac{p+1}{2}}E_\lambda
    \eeq
    for all $0<\mu<1$. Indeed, let $f\in H^1(\R^2)$ with $\|f\|_{L^2}^2=\lambda$, it is clear that $\|\sqrt{\mu}f\|_{L^2}^2=\mu\lambda$ and
    \[
    H(\sqrt{\mu}f )\geq \mu^{\frac{p+1}{2}}H(f),
    \]
     since $p>1$ and $0<\mu<1$. Thus, we have \eqref{ineq:additivity1}.

     Now, to prove \eqref{ineq:additivity}, let $\mu_1=\frac{\lambda_1}{\lambda_1+\lambda_2}$ and $\mu_2=\frac{\lambda_2}{\lambda_1+\lambda_2}$. Then $\mu_1+\mu_2=1$ and $0<\mu_1,\mu_2< 1$. By \eqref{ineq:additivity1} and Lemma \ref{lem:well-defined}, we have
     \beq\label{ineq:subadditivity}
      E_{\lambda_1}+  E_{\lambda_2} \geq \left( \mu_1^{\frac{p+1}{2}} +\mu_2^{\frac{p+1}{2}}\right) E_{\lambda_1+\lambda_2}\geq E_{\lambda_1+\lambda_2},
     \eeq
where we used the algebraic inequality
     \[
     1=(\mu_1+\mu_2)^q=\mu_1(\mu_1+\mu_2)^{q-1}+\mu_2(\mu_1+\mu_2)^{q-1}> \mu_1^q+\mu_2^q
     \]
      for all $q>1$.
If, in addition, $E_{\lambda_1+\lambda_2} <0$, then the second inequality in \eqref{ineq:subadditivity} becomes a strict inequality.
\end{proof}

Now, we are ready to present the properties of the map $\lambda \mapsto E_\lambda$, which are crucial in proving the threshold phenomenon.

\begin{proposition}\label{prop:cont}
 Assume $1< p < 5$.
 The map $\lambda \mapsto E_\lambda$ is decreasing and continuous on $(0, \infty)$.
\end{proposition}
\begin{proof}
Using Lemma \ref{lem:well-defined} and Proposition \ref{prop:subadditivity}, we have
\[
E_{\lambda_1}\geq E_{\lambda_1}+E_{\lambda_2}\geq E_{\lambda_1+\lambda_2}
\]
for all $\lambda_1, \lambda_2 >0$. Thus, the map $\lambda \mapsto E_\lambda$ is decreasing.

To show the continuity, fix $\lambda\in (0,\infty)$ and let $(\lambda_n)_n$ be any sequence of positive numbers converging to $\lambda$ as $n\to \infty$. We first show that
    \beq\label{limsup}
    \limsup_{n\to \infty} E_{\lambda_n}\leq E_\lambda.
    \eeq
    Take an arbitrary $f\in H^1(\R^2)$ with $\|f\|_{L^2}^2=\lambda$
    and set $f_n=\sqrt{\lambda_n/\lambda}\, f$, then $\|f_n\|_{L^2}^2=\lambda_n$ and
   \[
    H(f_n)=\f{\lambda_n}{\lambda}\f{\dav}{2} \|\nabla f\|_{L^2}^2
       -\left(\f{\lambda_n}{\lambda}\right)^{\frac{p+1}{2}}\frac{1}{p+1}\int_0 ^1 \!\int_{\R^2} |e^{ir\Delta} f(x)|^{p+1} dxdr \to H(f)
  \]
    as $n\to\infty$.
Thus, it is clear that
   \[
   \limsup_{n\to \infty} E_{\lambda_n}\leq \lim_{n\to\infty} H(f_n)=H(f)
   \]
which implies \eqref{limsup} since $f\in H^1(\R^2)$ is arbitrary with the constraint $\|f\|_{L^2}^2=\lambda$.\\
It remains to prove that
\[
\liminf_{n\to \infty} E_{\lambda_n} \geq E_\lambda.
\]
For each $n\in \N$, there exists $(g_n)_n\subset H^1(\R^2)$ with $\|g_n\|_{L^2}^2=\lambda_n$ such that
\beq\label{ineq:f_n}
H(g_n)< E_{\lambda_n}+\f{1}{n}.
\eeq
Then, by \eqref{limsup} and Lemma \ref{lem:well-defined},
\[
\limsup_{n\to \infty} H(g_n)\leq E_\lambda\leq 0.
\]
Moreover, using Lemma \ref{lem:Strichartz type 2} (i) and $\lambda_n\to\lambda$ as $n\to\infty$, we have
\[
H(g_n) \geq \f{\dav}{2}\|\nabla g_n\|_{L^2}^2 -C\|\nabla g_n\|_{L^2}^{\f{p-1}{2}}\lambda^{\f{p+3}{4}}
\]
for sufficiently large $n$.
Thus, it follows from the last two inequalities that $\|\nabla g_n\|_{L^2}$ stays bounded since $0< \f{p-1}{2}<2$.
If we define $\widetilde{g}_n=\sqrt{\lambda/\lambda_n} \, g_n$, then
$\|\widetilde{g}_n\|_{L^2}^2=\lambda$
and
\[
\begin{aligned}
E_\lambda &\leq H(\widetilde{g}_n) \\
&\leq H(g_n)+|H(\widetilde{g}_n)-H(g_n)|\\
&<  E_{\lambda_n}+\f{1}{n} + \f{\dav}{2}\left|\f{\lambda}{\lambda_n}-1\right|\|\nabla g_n\|_{L^2}^2
+\frac{1}{p+1}\left| \left(\f{\lambda}{\lambda_n}\right)^{\frac{p+1}{2}}-1\right|\|\nabla g_n\|_{L^2}^{\f{p-1}{2}}\lambda_n^{\f{p+3}{4}}
\end{aligned}
\]
by \eqref{ineq:f_n} and Lemma \ref{lem:Strichartz type 2} (i). Taking the limit inferior on both sides completes the proof.
\end{proof}

\subsection{Existence of minimizers} 
To prove the existence of minimizers, instead of directly using the method of the Lions' concentration compactness principle in \cite{Lions1, Lions2}, we follow a less technical approach in \cite{jeanjean}, which was introduced in \cite{Ikoma}.
For the vanishing scenario, we first recall the following result from \cite{ZGJT01} and provide its proof in Appendix \ref{appendix} for the reader’s convenience.
We denote by $B(y,1)$ the unit ball centered at $y\in\R^2$.
\begin{lemma}[Theorem 7.1 and Corollary 7.2 in \cite{ZGJT01}] \label{lem:localization}
Let $(f_n)_n$ be a bounded sequence in $H^1(\R^2)$ with $\|f_n\|_{L^2}=1$ . If $(f_n)_n$  is vanishing, i.e.,
\[
\sup_{y\in \R^2} \int_{B(y,1)}|f_n(x)|^2 dx \to 0
\]
as $n\to\infty$, then the sequence of the solutions $(e^{it\Delta}f_n)_n$ is also vanishing for each $t\in \R$, i.e.,
\[
\sup_{y\in \R^2} \int_{B(y,1)}|e^{it\Delta}f_n(x)|^2 dx \to 0
\]
as $n\to\infty$ for each $t\in \R$.
\end{lemma}

We provide a splitting result that is essential for the proof of existence. We refer to \cite[Lemma 4.3]{choi2023global} for the proof which adapts the method of Br\'{e}zis-Lieb \cite{Brezis}.

\begin{lemma}\label{lem:splitting}
Let $p>1$. If a bounded sequence $(f_n)_n$ in $H^1(\R^2)$ converges to $f\in H^1(\R^2)$ a.e. on $\R^2$, then
   \[
    \lim_{n\to \infty} \int_0^1 \!\int_{\R^2} \big| |e^{ir\Delta} f_n(x)|^{p+1} -|e^{ir\Delta}(f_n-f)(x)|^{p+1} - |e^{ir\Delta} f(x)|^{p+1}\big| dxdr =0.
  \]
\end{lemma}

\bigskip

Now, we are ready to prove the existence of minimizers when $E_{\lambda}<0$.

\begin{theorem}\label{thm:minimizer}
   Assume $1<p<5$. If $E_\lambda<0$, then there exists a minimizer of \eqref{eq:min_problem}. Moreover, this minimizer solves the Euler-Lagrange equation
    \beq\label{eq:E-L_2}
    -\omega f = -\dav \Delta f -\int_0^1 e^{-ir\Delta}\left(|e^{ir\Delta}f|^{p-1}e^{ir\Delta}f\right)dr
    \eeq
    for some positive Lagrange multiplier $\omega>-2E_{\lambda}/\lambda$.
\end{theorem}

\begin{proof}
  Given $\lambda>0$, let $(f_n)_n\subset H^1(\R^2)$ be a minimizing sequence of \eqref{eq:min_problem}, that is, $H(f_n)\to E_\lambda$ as $n\to\infty$ and $\|f_n\|_{L^2}^2=\lambda$. Then,  $(f_n)_n$ is bounded in $H^1(\R^2)$. Indeed, by Lemma \ref{lem:Strichartz type 2} (i), we see that
  \[
  H(f_n)\geq \f{\dav}{2}\|\nabla f_n\|_{L^2}^2 -C\lambda ^{\f{p+3}{4}}\|\nabla f_n\|_{L^2}^{\f{p-1}{2}}
  \]
  Since $0< \frac{p-1}{2} <2$ and $H(f_n)\to E_\lambda$ as $n\to\infty$, this immediately implies that $\|\nabla f_n\|_{L^2}$ is bounded, and consequently,
  $\|f_n\|_{H^1}$ is bounded.

  We will show that $(f_n)$ is non-vanishing.
  Since $H(f_n)\to E_\lambda<0$, we have
  \[
  \liminf _{n\to \infty}\int_0^1 \!\int_{\R^2} |e^{ir\Delta}f_n(x)|^{p+1}dx\,dr>0,
  \]
  Then, up to a subsequence, there exist $c>0$ and $r_0\in [0,1]$ which are independent of $n$ such that
  \[
  \int_{\R^2}|e^{ir_0\Delta}f_n(x)|^{p+1} dx >c.
  \]
  Using the Sobolev estimate \cite{Lions1},
  for any $g\in H^1(\R^2)$
  \[
  \|g\|_{L^{p+1}}^{p+1} \lesssim \left( \sup_{y\in\R^2} \int_{B(y,1)} |g(x)|^2 dx \right) \|g\|_{H^1}^2,
  \]
  we have
\[
 c<\|e^{ir_0\Delta}f_n\|_{L^{p+1}}^{p+1} \lesssim \left( \sup_{y\in\R^2} \int_{B(y,1)} |e^{ir_0\Delta}f_n(x)|^2 dx \right) \|f_n\|_{H^1}^2 .
\]
Therefore, we see that for some constant $C>0$
\[
  \sup_{y\in\R^2} \int_{B(y,1)} |e^{ir_0 \Delta}f_n(x)|^2 dx > C
  \]
since $\|f_n\|_{H^1}$ is bounded.
Thus,  by Lemma \ref{lem:localization}, $(f_n)_n$ is non-vanishing. Therefore,
 there exists
  a sequence $(y_n)_n\subset \R^2$ such that
   \beq\label{vanishing}
 \sup_{y_n\in \R^2} \int_{B(0,1)} |f_n(x+y_n)|^2dx>C.
  \eeq

Let $h_n(x)=f_n(x+y_n)$ for all $x\in \R$. Then $(h_n)_n$ is also a minimizing sequence for $E_\lambda$ and $(h_n)_n$ is bounded in $H^1(\R^2)$. Thus, there exists $f\in H^1(\R^2)$ such that up to a subsequence, $h_n\rightharpoonup f$ in $H^1(\R^2)$ and $h_n \to f$ a.e. on $\R^2$ as $n\to \infty$.
Since the embedding $H^1(B(0,1)) \hookrightarrow L^2(B(0,1))$ is compact, $(h_n)_n$ converges strongly to $f$ on $L^2(B(0,1))$ and therefore, by \eqref{vanishing}, $f$ is nontrivial.

Setting $t_n:=\|h_n-f\|_{L^2}^2$ for each $n\in \N$ and $\lambda':=\|f\|_{L^2}^2$, it is clear that $0<\lambda' \le \lambda$ and that
\[
\lim_{n\to \infty} t_n = \lambda - \lambda'
\]
since $h_n\rightharpoonup f\neq 0$ in $L^2(\R^2)$.
We claim that $\lambda=\lambda'$. Suppose to the contrary that $\lambda'<\lambda$, then $\lim_{n\to \infty} t_n>0$. From the definition of $E_{t_n}$, we have
\beq\label{ineq:E_tn}
\liminf_{n\to \infty}H(h_n-f)\geq \lim_{n\to\infty} E_{t_n} =E_{\lambda-\lambda'},
\eeq
where we used the continuity of the map $\lambda \mapsto E_\lambda$.
On the other hand, since $h_n \rightharpoonup f$ in $H^1(\R^2)$,
\[
\lim_{n\to \infty} \left( \|\nabla h_n\|_{L^2}^2 -\|\nabla (h_n-f)\|_{L^2}^2 -\|\nabla f\|_{L^2}^2 \right)=0
\]
which together with Lemma \ref{lem:splitting}, we have
\beq\label{eq:splitting1}
\lim_{n\to\infty} H(h_n)=\lim_{n\to\infty} H(h_n-f)+H(f)
\eeq
since $h_n \to f$ a.e. on $\R^2$ as $n\to \infty$.
Thus, using \eqref{eq:splitting1}, \eqref{ineq:E_tn}, and \eqref{Strict Sub-additivity} with $E_\lambda<0$, we obtain a contradiction as
\[
E_\lambda = \lim_{n\to\infty} H(h_n)\geq E_{\lambda-\lambda'}+H(f) \geq E_{\lambda-\lambda'} + E_{\lambda'} >E_\lambda,
\]
which proves the claim.

Thus,
\[
\|f\|_{L^2}^2=\lambda=\lim_{n\to\infty}\|h_n\|_{L^2}^2
\]
and therefore $(h_n)$ converges strongly to $f$ in $L^2(\R^2)$.
Moreover, we also have the weak sequential lower semi-continuity of the $H^1$ norm, that is,
\[
\|f\|_{H^1}\leq \liminf_{n\to\infty}\|h_n\|_{H^1},
\]
which implies
\beq\label{kinetic}
\|\nabla f\|^2_{L^2}\leq \liminf_{n\to\infty}\|\nabla h_n\|^2_{L^2}.
\eeq
Since $(h_n)$ is bounded in $H^1(\R^2)$ and $(h_n)$ converges strongly to $f$ in $L^2(\R^2)$, by Lemma \ref{lem:Lipshhitz}, we have
\[
\begin{aligned}
\lim_{n\to \infty}
\int_0 ^1 \|e^{ir\Delta} h_n \|_{L^{p+1}}^{p+1}dr = \int_0 ^1 \|e^{ir\Delta} f \|_{L^{p+1}}^{p+1} dr.
\end{aligned}
\]
This together with \eqref{kinetic} implies
\[
E_\lambda \leq H(f) \leq \lim_{n\to \infty} H(h_n) =E_\lambda,
\]
which shows that $f$ is a minimizer for $E_\lambda$ since $\|f\|_{L^2}^2=\lambda$.

By a standard argument in the calculus of variations, one can show that every minimizer is a weak solution of the associated Euler–Lagrange equation \eqref{eq:E-L_2} for some Lagrange multiplier $\omega$,
and therefore it is a strong solution of \eqref{eq:E-L_2} since
\[
g\mapsto \int_0^1\e^{-ir\Delta}\left(|e^{ir\Delta}g|^{p-1} e^{ir\Delta}g\right)dr
\]
maps $H^1(\R^2)$ into itself. Lastly, since every minimizer  solves \eqref{eq:E-L_2}, we have
\[
-\omega \lambda =2H(f)-\left(1-\frac{2}{p+1}\right)\int_0^1\|e^{ir\Delta}f\|_{L^{p+1}}^{p+1}dr< 2H(f) =2E_\lambda
\]
for any minimizer $f$, which is $\omega>-2E_{\lambda}/\lambda>0$.
\end{proof}

\subsection{Proof of Theorem \ref{thm:existence}} \label{sec:Proof of Theorem 1.1}

Motivated by Theorem \ref{thm:minimizer}, we define the threshold as follows.
\begin{definition}[Threshold]
\[
\lambda_{cr}:=\inf\{\lambda>0 : E_\lambda<0\}.
\]
\end{definition}

Then, we have
\begin{lemma}\label{lem:finite}
\begin{theoremlist}
  \item If $1<p<3$, then $\lambda_{cr}=0$.
  \item If $3\leq p<5$, then $0<\lambda_{cr}<\infty$.
\end{theoremlist}
\end{lemma}
\begin{proof}
If $1<p<3$, it immediately follows from Lemma \ref{lem:well-defined} that $\lambda_{cr}=0$. If $3\leq p<5$, using Lemma \ref{lem:Strichartz type 2} (ii), we have
\[
H(f)\geq \|\nabla f\|_{L^2}^2 \left( \frac{\dav}{2}-\frac{C\lambda^{\frac{p-1}{2}}}{p+1}\right)
\]
for all $f\in H^1(\R^2)$ with $\|f\|_{L^2}^2=\lambda$. Thus, for all $\lambda$ small enough, we have $E_\lambda \ge 0$ which together with Lemma \ref{lem:well-defined} yields $E_\lambda = 0$. Therefore, $\lambda_{cr}>0$.

To prove that $\lambda_{cr}$ is finite, it is enough to find a suitable test function with negative energy $H$.
Consider the Gaussian test function $g_{\sigma_0}$ with $\sigma_0>0$ in \eqref{gaussian} which has  $\|g_{\sigma_0}\|_{L^2}^2=\lambda$.
In special, if we set $\sigma_0=\lambda$ in \eqref{eq:gaussian_H}, then
\[
H(g_{\lambda})= \dav \left(1- \f{2^{\f{p+1}{2}}\pi^{\f{1-p}{2}} }{\dav  (p+1)^2} \, \lambda  \int_0^1 \left( \f{1}{1+(4r/\lambda)^2} \right)^{\f{p-1}{2}} dr\right).
\]
Choosing $\lambda>0$ large enough, we have $H(g_{\lambda})<0$.
\end{proof}

Now we present

\begin{proof}[Proof of Theorem \ref{thm:existence}]
First, we consider the case when $\lambda>\lambda_{cr}$. Since the map $\lambda \mapsto E_\lambda$ is decreasing on $(0,\infty)$, it is clear that $E_\lambda<0$.
Thus, it follows from Theorem  \ref{thm:minimizer} that there is a minimizer for \eqref{eq:min_problem} which solves \eqref{eq:EL} for some $\omega>0$.

Next, we consider the case when $0<\lambda<\lambda_{cr}$. Then, by Lemma \ref{lem:well-defined}, it is clear that $E_\lambda=0$. Moreover, there is no minimizer for \eqref{eq:min_problem}. Indeed, suppose to the contrary that $f\in H^1(\R^2)$ with $\|f\|_{L^2}^2=\lambda$ is a minimizer for \eqref{eq:min_problem}. Choose $\widetilde{\lambda}$ so that $ \lambda<\widetilde{\lambda}<\lambda_{cr}$ and set $\widetilde{f}=\sqrt{\widetilde{\lambda}/\lambda}\, f$, then it follows that $\|\widetilde{f}\|_{L^2}^2=\widetilde{\lambda}$ and
\[
\begin{aligned}
0=E_{\widetilde{\lambda}}\leq H(\widetilde{f}) &=\f{\widetilde{\lambda}}{\lambda}H(f)+\f{\widetilde{\lambda}}{\lambda}\left(1-\left(\f{\widetilde{\lambda}}{\lambda}\right)^{\f{p-1}{2}}\right)\f{1}{p+1} \int_0^1\|e^{ir\Delta} f\|_{L^{p+1}}^{p+1}dr\\
&<\f{\widetilde{\lambda}}{\lambda}H(f)=\f{\widetilde{\lambda}}{\lambda}E_\lambda=0
\end{aligned}
\]
which is a contradiction.

It remains to consider the case when $\lambda=\lambda_{cr}>0$. In this case, we only consider $3 \le p<5$ by Lemma \ref{lem:finite}. Since the map $\lambda \mapsto E_\lambda$ is continuous on $(0,\infty)$ and $E_\lambda=0$ for all $\lambda <\lambda_{cr}$ by the second case, we have $E_{\lambda_{cr}} =0$.
To prove the existence of a minimizer for $E_{\lambda_{cr}}=0$,
define $(\lambda_n)_n\subset\R$ by $\lambda_n=\lambda_{cr}+1/n$ for any $n\in \N$.
Since $\lambda_n >\lambda_{cr}$, by the first case, we have $E_{\lambda_n}<0$ for all $n\in \N$, and therefore, by Theorem \ref{thm:minimizer}, we can choose a sequence of minimizers $(f_n)_n\subset H^1(\R^2)$ with $\|f_n\|_{L^2}^2=\lambda_n$ such that
\beq\label{eq:f_n}
H(f_n)=E_{\lambda_n}<0.
\eeq
Since $\lambda_n\to \lambda_{cr}$ as $n\to \infty$, by the continuity of the map $\lambda \mapsto E_\lambda$, we have $E_{\lambda_n}\to E_{\lambda_{cr}}=0$.
Thus, $\| f_n\|_{H^1}$ is bounded, since, by Lemma \ref{lem:Strichartz type 2} (i),
  \[
0>  E_{\lambda_n} = H(f_n)\geq \f{\dav}{2}\|\nabla f_n\|_{L^2}^2 -\f{C}{p+1}\lambda_1 ^{\f{p+3}{4}}\|\nabla f_n\|_{L^2}^{\f{p-1}{2}}
  \]
and $p<5$.
Moreover, it follows from Lemma \ref{lem:m_lambda} below and \eqref{eq:f_n} that up to a subsequence $\lim_{n\to\infty} \|\nabla f_n\|_{L^2}^2>0$.
From the fact that $\lim_{n\to\infty} H(f_n)=0$, it is clear that
\[
\lim_{n\to \infty} \int_0^1 \|e^{ir\Delta}f_n\|_{L^{p+1}}^{p+1}dr>0.
\]
Thus, we can apply the  same argument as in the proof of Theorem \ref{thm:minimizer} to see  $(f_n)$ is non-vanishing, and therefore, there exist  a sequence $(y_n)_n$ and a nontrivial $f\in H^1(\R^2)$ such that up to a subsequence, $f_n(\cdot+y_n)\rightharpoonup f$ in $H^1(\R^2)$ and $f_n(\cdot+y_n) \to f$ a.e. on $\R^2$ as $n\to \infty$.


If we set $h_n=f_n(\cdot +y_n)$ and $\lambda':=\|f\|_{L^2}^2$, then it follows from $\lambda_n\to \lambda_{cr}$ as $n\to \infty$ that $\lambda'\in (0, \lambda_{cr}]$ and
\[
\lim_{n\to \infty} \|h_n-f\|_{L^2}^2= \lambda_{cr} - \lambda'\in [0, \lambda_{cr}).
\]
It remains to show that $\|f\|_{L^2}^2=\lambda_{cr}$ since it implies the strong convergence of $(h_n)$ to $f$ in $L^2(\R^2)$.
Suppose to the contrary that $\lambda'<\lambda_{cr}$. Then, by the second case, $E_{\lambda'}=0$ and $E_{\lambda_{cr}-\lambda'}=0$. Thus, we have
\beq\label{eq:critical}
H(f)\geq E_{\lambda'}=0 \quad \mbox{and}\quad \liminf_{n\to\infty}H(h_n-f)\geq E_{\lambda_{cr}-\lambda'}=0
\eeq
where we used the continuity of the map $\lambda \mapsto E_\lambda$ in the latter.
On the other hand, we use the same argument as in \eqref{eq:splitting1} to obtain
\[
\lim_{n\to\infty} H(h_n)=\lim_{n\to\infty} H(h_n-f) +H(f).
\]
Since $H(h_n)$ converges to $E_{\lambda_{cr}}=0$ as $n\to \infty$, we have, by \eqref{eq:critical},
\[
 H(f)=0=E_{\lambda'} 
\]
which contradicts to the second case. Thus, $\lambda'=\lambda_{cr}$.
As in the proof of Theorem \ref{thm:minimizer}, we have that $f$ is a minimizer which solves \eqref{eq:EL} for some Lagrange multiplier $\omega>-2E_{\lambda_{cr}}/\lambda_{cr}=0$.

Finally, to prove the last assertion, (v), we define
\[
\alpha_p=\left(\frac{\dav (p+1)}{2\calC_p}\right)^{\frac{2}{p-1}}, \quad 3\leq p<5,
\]
where $\calC_p$ is the best constant for \eqref{ineq:lambda_cr}, i.e.,
\[
\calC_p=\sup_{f\in H^1(\R^2)}\frac{\int_0^1 \|e^{ir\Delta} f\|_{L^{p+1}}^{p+1}dr }{\|\nabla f\|_{L^2}^2 \|f\|_{L^2}^{p-1}}.
\]
If $\lambda\leq\alpha_p $, then by \eqref{ineq:Strichartz type 4}, we have that
\[
H(f)\geq \|\nabla f\|_{L^2}^2\left(\frac{\dav}{2}-\frac{\calC_p}{p+1}\lambda^{\frac{p-1}{2}}\right) \geq 0
\]
for any $f\in H^1(\R^2)$ with $\|f\|_{L^2}^2=\lambda$. Thus, $E_\lambda\geq 0$, and therefore, by Lemma \ref{lem:well-defined}, $E_\lambda=0$, which implies $\alpha_p\leq \lambda_{cr}$ due to the first assertion. \\
If $\lambda>\alpha_p$, that is, $\frac{\dav(p+1)}{2\lambda^{(p-1)/2}}< \calC_p$, then there is $f_0\in H^1(\R^2)$ such that
\beq\label{f_0}
\frac{\dav(p+1)}{2\lambda^{\frac{p-1}{2}}}< \frac{\int_0^1 \|e^{ir\Delta} f_0\|_{L^{p+1}}^{p+1}dr }{\|\nabla f_0\|_{L^2}^2 \|f_0\|_{L^2}^{p-1}} \leq \calC_p
\eeq
If we set $\widetilde{f_0}=\sqrt{\lambda}\|f_0\|^{-1}_{L^2}f_0$, then $\|\widetilde{f_0}\|_{L^2}^2=\lambda$ and
\[
E_\lambda\leq
H(\widetilde{f_0})=\frac{\lambda}{\|f_0\|_{L^2}^2}\left(\frac{\dav}{2}\|\nabla f_0\|_{L^2}^2 -\frac{\lambda^{\frac{p-1}{2}}}{(p+1)\|f_0\|_{L^2}^{p-1}}
\int_0^1 \|e^{ir\Delta} f_0\|_{L^{p+1}}^{p+1}dr\right).
\]
It follows from \eqref{f_0} that $H(\widetilde{f_0})<0$, which shows $\alpha_p \geq \lambda_{cr}$ by the first assertion, again. Therefore, $\lambda_{cr}=\alpha_p$.
\end{proof}

\begin{lemma}\label{lem:m_lambda}
Let $1<p<5$. For any $\lambda>0$, there exists $\beta=\beta(\lambda)>0$ such that  if $f\in H^1(\R^2)$ satisfying both $\|f\|_{L^2}^2\leq \lambda$ and $\|\nabla f\|_{L^2}^2<\beta(\lambda)$, then
\[
H(f)> \frac{\dav}{4}\|\nabla f\|_{L^2}^2.
\]
\end{lemma}
\begin{proof}
   If we let $\beta=\left( \frac{\dav (p+1)}{4C}\lambda^{-\frac{p+3}{4}}\right)^{\frac{5-p}{2}}$, where $C$ is the best constant of \eqref{ineq:Strichartz type 2}, then this lemma immediately follows from Lemma \ref{lem:Strichartz type 2}.
\end{proof}

\section{Proof of Theorem \ref{thm:critical}} \label{sec:Proof of Theorem 1.3}
First, we consider the following {\it global in $r$} inequality
\beq\label{ineq:global}
   \int_\R\|e^{ir\Delta}f\|_{L^6}^6 dr\lesssim \|\nabla f\|_{L^2}^{2}\,\|f\|_{L^2}^{4}.
\eeq
We denote the best constant for \eqref{ineq:global} by $\calC(\R)$. Then $\calC(\R)<\infty$. Indeed,
we use the inequality \eqref{ineq:well-known} and the Cauchy-Schwarz inequality to have
\[
\begin{aligned}
\int_\R\|e^{ir\Delta}f\|_{L^6}^6dr &\leq \|f\|_{L^2}^2 \int_\R\|e^{ir\Delta}f\|_{L^\infty}^4dr \\
& \lesssim \|f\|_{L^2}^2 \int_\R \|e^{ir\Delta}\nabla f\|_{L^4}^{2} \, \|e^{ir\Delta}f\|_{L^4}^2 \, dr\\
&\le \|f\|_{L^2}^2 \left(\int_\R  \|e^{ir\Delta}\nabla f\|_{L^4}^{4}\,dr\right)^{\f{1}{2}}\left(\int_\R \|e^{ir\Delta}f\|_{L^4}^{4}\,dr\right)^{\f{1}{2}}
\end{aligned}
\]
which, together with the Strichartz inequality, deduces \eqref{ineq:global}.

\begin{proposition}[Critical element for inequality \eqref{ineq:global}]\label{prop:maximizer}
We define the Weinstein functional associated with the inequality \eqref{ineq:global} by
\beq\label{Weinstein}
W (f)=\f {\int_\R \|e^{ir\Delta}f \|_{L^{6}}^{6} dr}{ \|\nabla f\|_{L^2}^{2} \|f\|_{L^2}^{4}}. \notag
\eeq
Then, the variational problem
\beq\label{variational on R}
\calC(\R)=\sup_{f\in H^1(\R^2)}  W(f) \notag
\eeq
admits a maximizer $Q\in H^1(\R^2)$ which solves the Euler-Lagrange equation
\beq\label{eq:EL_global}
    -\Delta Q+ Q- \int_\R e^{-ir\Delta}\big(|e^{ir\Delta}Q|^5e^{ir\Delta}Q\big)dr=0.
\eeq
Moreover, the maximizer $Q$ satisfies
\beq\label{eq:normidentities}
\|Q\|_{L^2}^2=\sqrt{ \f{3}{\calC(\R)}}, \quad \|\nabla Q\|^2_{L^2}=\f{1}{2}\|Q\|^2_{L^2}, \quad \mbox{and}
\quad
\int_\R\|e^{ir\partial_x^2}Q \|_{L^{6}}^{6}dr=\f{3}{2}\|Q\|_{L^2}^2.
\eeq
\end{proposition}

To prove this proposition, we closely follow the argument of \cite[Theorem 1.1]{choi2023global} which considered the one-dimensional case.

\begin{proof}
We first can show that there exists a maximizer $g\in H^1(\R^2)$ for $\calC(\R)$ such that $\|g\|_{L^2}=\|\nabla g\|_{L^2}=1$, which solves the Euler-Lagrange equation
    \beq\label{eq:EL_g}
    \Delta g - 2g+\f{3}{\calC(\R)}\int_\R e^{-ir\Delta}\big(|e^{ir\Delta}g|^5e^{ir\Delta}g\big)dr =0.
    \eeq
Its proof follows a Lions' concentration compactness principle \cite{Lions1, Lions2}, and it can proceed similar to that of Proposition 4.2 in \cite{choi2023global}, which employs the profile decomposition, see, e.g., \cite[Proposition 3.4]{Guevara}.
    Moreover, it follows that
    \[
    \int_\R \|e^{ir\Delta}g\|_{L^6}^6dr =\calC(\R).
    \]

    Next, we define $Q(x)=\mu g(\lambda x)$ with $\mu=  \left(3/ 4\calC(\R) \right)^{1/4}$ and $\lambda=1/ \sqrt{2}$.
    Since the functional $W(\cdot)$ is invariant under scaling and multiplication by a constant, $Q$ is also a maximizer for $\calC(\R)$.
    Moreover, since $g$ solves \eqref{eq:EL_g}, by direct calculations, $Q$ solves \eqref{eq:EL_global}.
     Using $\|g\|_{L^2}=\|\nabla g\|_{L^2}=1$, we have
    \[
    \|Q\|_{L^2}^2=\mu^2\lambda^{-2}=\sqrt{ \f{3}{\calC(\R)}}, \quad \|\nabla Q\|^2_{L^2}=\mu^2=\f{1}{2}\sqrt{ \f{3}{\calC(\R)}}=\f{1}{2}\|Q\|^2_{L^2}.
    \]
Furthermore,
\[
\int_\R \|e^{ir\partial_x^2}Q \|_{L^{6}}^{6}dr=\calC(\R)\|Q\|_{L^2}^4\|\nabla Q\|_{L^2}^2=\f{3}{2}\|Q\|_{L^2}^2.
    \]
\end{proof}

  \begin{remarks}
The uniqueness (up to symmetries) of the critical element for the Weinstein functional $W(\cdot)$ remains unknown.  However, from \eqref{eq:normidentities}, the important norm quantities can be expressed only in terms of $\calC(\R)$, and they do not depend on a possibly non-unique profile $Q$, see \cite[Lemma 4.5]{choi2023global} for proof.
\end{remarks}

\begin{proposition}\label{prop:no maximizer}
   There is no maximizer for the  variational problem
   \[
   \calC_5=\sup_{f\in H^1(\R^2)} \f{ \int_{0}^1\|e^{ir\Delta}f \|_{L^6}^{6} dr }{ \|\nabla f\|_{L^2}^{2} \|f\|_{L^2}^{4}}.
   \]
\end{proposition}
\begin{proof}
For any interval $I\subset \R$,
denote the functional
\[
W_{I}(f)= \f{ \int_{I}\|e^{ir\Delta}f \|_{L^6}^{6} dr }{ \|\nabla f\|_{L^2}^{2} \|f\|_{L^2}^{4}}
\]
and the variational problem
\[
\calC(I)=\sup_{f\in H^1(\R^2)}  W_I(f).
\]
First, we prove that
$\calC([0,1])= \calC([-1,1])$. Indeed, for any $f\in H^1(\R)$, we set $\widetilde{f}=f(\sqrt{2}\,\cdot)$, then the function $e^{-\frac{i}{2}\Delta }\widetilde{f}$ satisfies
\[
\|e^{-\frac{i}{2}\Delta }\widetilde{f}\|_{L^2}^2=\frac{1}{2}\|f\|_{L^2}^2, \quad \|\nabla e^{-\frac{i}{2}\Delta }\widetilde{f}\|_{L^2}^2= \|\nabla f\|_{L^2}^2,
\]
and
\[
\begin{aligned}
    \int_0^1 \|e^{ir\Delta} e^{-\frac{i}{2}\Delta}\widetilde{f}\|_{L^6}^6 dr & = \int_0^1  \|(e^{2i(r-\frac{1}{2})\Delta} f)(\sqrt{2}\,\cdot)\|_{L^6}^6 dr\\
    & =\frac{1}{4}\int_{-1}^{1}\|e^{ir\Delta}f\|_{L^6}^6 dr.
\end{aligned}
\]
Thus, it follows that
\[
W_{[0,1]}(e^{-\frac{i}{2}\Delta }\widetilde{f})= W_{[-1, 1]}(f).
\]

Now, to show that $\calC_5=\calC([0,1])$ is not achieved, suppose to the contrary that $f\in H^1(\R^2)$ is a maximizer for $\calC([0,1])$. Then
\[
\calC([-1,1]) \geq W_{[-1,1]}(f)
=\f{ \int_{-1} ^0\|e^{ir\Delta}f \|_{L^6}^{6} dr }{ \|\nabla f\|_{L^2}^{2} \|f\|_{L^2}^{4}}+ W_{[0,1]}(f) =\f{ \int_{-1} ^0\|e^{ir\Delta}f \|_{L^6}^{6} dr }{ \|\nabla f\|_{L^2}^{2} \|f\|_{L^2}^{4}}+\calC([0,1]).
\]
 The fact that $\calC([-1,1]) = \calC([0,1])$ yields $\int_{-1}^0 \|e^{ir\Delta}f \|_{L^6}^6 =0$, which contradicts to $\|f\|_{L^2}\neq 0$.
 Therefore, the variational problem $\calC([0,1])$ does not have a maximizer.
\end{proof}

Using the last two propositions, we prove Theorem \ref{thm:critical}.
\begin{proof}[Proof of Theorem \ref{thm:critical}]
Note that, by \eqref{eq:G-Ntype-0},
\beq\label{ineq: critical}
\begin{aligned}
    H(f) &=\f{\dav}{2}\|\nabla f\|_{L^2}^2-\f{1}{6}\int_0^1 \|e^{ir\Delta}f \|_{L^{6}}^{6}dr\\
    &\geq  \|\nabla f\|_{L^2}^2 \left( \f{\dav}{2} - \f{\calC_5}{6} \lambda^2\right)
\end{aligned}
\eeq
for all $f\in H^1(\R^2)$ with $\|f\|_{L^2}^2=\lambda$, where $\calC_5$ is the best constant for \eqref{eq:G-Ntype-0}.  We set
\[
\lambda_{cr}:=\sqrt{\frac{3\dav} {\calC_5}}.
\]
If $\lambda \leq \lambda_{cr}$, then, by \eqref{ineq: critical}, $H(f)\geq 0$ for all $f\in H^1(\R^2)$ with $\|f\|_{L^2}^2=\lambda$, which implies $E_\lambda\geq 0$.
On the other hand,  by the same argument as in Lemma \ref{lem:well-defined}, we have $E_\lambda
\leq 0$. Therefore,  $E_\lambda= 0$.

Now, we assume that $\lambda>\lambda_{cr}$, i.e., $\lambda ^2 > 3\dav/\calC_5$. If we define
\[
\widetilde{Q}:=\widetilde{Q}(\beta)=\mu e^{-\f{i}{2}\Delta}Q_{\sqrt{2}\beta} \quad\mbox{with}\quad \mu^2 =2\lambda\beta^2  \|Q\|_{L^2}^{-2},
\]
where $Q_{\sqrt{2}\beta}=Q(\sqrt{2}\beta \cdot)$ and $Q$ is the extrimizer for
\eqref{ineq:global}, given in Proposition \ref{prop:maximizer}.
Then we calculate
\[
\|\widetilde{Q}\|_{L^2}^2=\frac{\mu^2}{2\beta^2}=\lambda\quad \mbox{and} \quad
\|\nabla \widetilde{Q}\|_{L^2}^2=\mu^2 \|\nabla Q\|_{L^2}^2= \lambda \beta^2,
\]
where we used the fact that $2\| \nabla Q\|_{L^2}^2 =\|Q\|_{L^2}^2$, by \eqref{eq:normidentities}.
Moreover,
\[
\int_0^1 \|e^{ir\Delta} \widetilde{Q}\|_{L^6}^{6} dr =
\frac{\mu^{6}} {(\sqrt{2}\beta)^{4}} \int_{-\beta^2}^{\beta^2}\|e^{ir\Delta} Q\|_{L^{6}}^{6} dr= 2\beta^2 \lambda^3 \|Q\|_{L^2}^{-6}\int_{-\beta^2}^{\beta^2}\|e^{ir\Delta} Q\|_{L^{6}}^{6} dr.
\]
Thus, we have
\beq\label{H_Q}
\begin{aligned}
    H(\widetilde{Q}) =
    \lambda \beta^2  \left( \f{\dav}{2} - \f{\lambda^2} {3} \|Q\|_{L^2}^{-6}\int_{-\beta^2}^{\beta^2}\|e^{ir\Delta} Q\|_{L^{6}}^{6} dr\right).
\end{aligned}
\eeq
Since $Q$ is the extremizer for \eqref{ineq:global}, one sees that
\[
\int_{-\beta^2}^{\beta^2}\|e^{ir\Delta} Q\|_{L^{6}}^{6} dr \to \int_{\R} \|e^{ir\Delta}Q\|_{L^{6}}^{6} dr =
 \calC(\R) \|\nabla Q\|_{L^2}^2\|Q\|_{L^2}^4 = \f{\calC(\R)}{2}\|Q\|_{L^2}^{6}
\]
as $\beta \to \infty$, where we again used the fact that $2\| \nabla Q\|_{L^2}^2 =\|Q\|_{L^2}^2$.
It is clear that $\calC_5\leq \calC(\R)$, and so
\[
\frac{\dav}{2}-\frac{\lambda^2}{6}\calC(\R)\leq \frac{\dav}{2}-\frac{\lambda^2}{6}\calC_5<0.
\]
Thus, by \eqref{H_Q}, $H(\widetilde{Q})\to -\infty$ as $\beta\to \infty$, which concludes that $E_\lambda=-\infty$. 

It remains to show that $E_\lambda$ is not achieved when $0<\lambda\leq \lambda_{cr}$. Let $0<\lambda^2\leq 3\dav/\calC_5$. Assume by contradiction that there exists a minimizer $f$ for $E_\lambda$, that is, $0=E_\lambda=H(f)$. Then
\[
\f{\dav}{2}\|\nabla f\|_{L^2}^2=\f{1}{6}\int_0^1 \|e^{ir\Delta}f\|_{L^{6}}^{6}  dr, \quad \|f\|_{L^2}^2=\lambda.
\]
Due to the best constant $\calC_5$ for \eqref{eq:G-Ntype-0} and the condition of $\lambda$, we have
\[
\calC_5\geq \f{\int_0^1\|e^{ir\Delta}f\|_{L^{6}}^{6}  dr}{ \|\nabla f\|_{L^2}^2 \|f\|_{L^2}^4}= \f{3\dav\|\nabla f\|_{L^2}^2}{ \|\nabla f\|_{L^2}^2 \|f\|_{L^2}^4}=\f{3\dav}{\lambda^2}\geq \calC_5,
\]
which yields $f$ is a maximizer for $\calC_5$.
However, it follows from Proposition \ref{prop:no maximizer} that $\calC_5$ does not have a maximizer, which is a contradiction.
\end{proof}

\begin{remark} \label{rem:unbounded}
If $p>5$, then $E_\lambda=-\infty$ for all $\lambda>0$. Indeed, if we employ the Gaussian test function $g_{\sigma_0}$ from \eqref{gaussian}, then $\|g_{\sigma_0}\|^2_{L^2}=\lambda$. Moreover, it follows from \eqref{eq:gaussian_H} that
\[
H(g_{\sigma_0})= \f{\dav \lambda}{\sigma_0}\left(1- C(p, \lambda, \dav) \sigma_0^{\f{5-p}{2}}  \int_0^{4/\sigma_0} \left( \f{1}{1+s^2} \right)^{\f{p-1}{2}} ds\right).
\]
Since $p>5$, taking the limit as $\sigma_0 \downarrow 0$ deduces $H(g_{\sigma_0})\to -\infty$.

\end{remark}

\appendix
\section{ Bound on localization in the linear Schr\"odinger equation}\label{appendix}
In this section, we prove Lemma \ref{lem:localization} by following the approach from \cite{ZGJT01}, where the proof was provided for the one-dimensional case only.

\begin{proof}[Proof of Lemma \ref{lem:localization}]
Let us write $u_n(t,x):=e^{it\Delta}f_n(x)$, then
we have
   \[
   \frac{d}{dt}\int_{B(0,R)} |u_n(t,x)|^2 dx=-2\im \int_{\partial B(0,R)}\overline{u_n(t,x)}(\nabla u_n(t,x)\cdot n) ds
   \]
for all $R>0$, where $n$ is the outer normal vector on the boundary of $B(0,R)$. Thus, integrating the above with respect to $t$ gives
    \beq\label{eq:appendix}
 \int_{B(0,R)} |u_n(t,x)|^2 dx-\int_{B(0,R)} |u_n(0,x)|^2 dx
=-\int_0 ^t2\im  \int_{\partial B(0,R)}\overline{u_n(\tau,x)}(\nabla u_n(\tau,x)\cdot n) ds d\tau.
    \eeq

Without loss of generality, we assume that $t>0$. Note that there is $C>0$ such that $\|f_n\|_{H^1}\le C$ for all $n$.
It suffices to show that
  \beq\label{eq:claim-app}
 \veps_n(t) \leq \sqrt{6}\,\veps_n(0)+\sqrt[3]{18\,\veps^2_n(0)+6\sqrt{6}\,C\veps_n(0)t},
\eeq
for all $n\in \N$ and $t>0$, where
\[
\veps_n(t):=\left(\sup_{y\in \R^2}\int_{B(y,1)}|u_n(t,x)|^2dx \right)^{1/2}.
\]

Now we fix $t>0$, then by the translational invariance, we may and do assume that $u_n(t, \cdot)$ is centered at the origin, i.e.,
    \[
    \veps^2_n(t)=\int_{B(0,1)}  |u_n(t,x)|^2 dx .
    \]
Then, for all $R\geq 1$, it is clear that
    \beq\label{eq:lower bound for L2 norm}
    \int_{B(0,R)}|u_n(t,x)|^2 dx \geq \veps^2_n(t)
    \eeq
and, moreover, we have
    \beq\label{eq:upper bound for L2 norm}
    \int_{B(0,R)} |u_n(0,x)|^2 dx \leq 12R^2 \veps^2_n(0).
    \eeq
Indeed, using the least natural number $J$ such that
\[
J\sqrt{2}+\frac{1}{\sqrt{2}} \ge R,
\]
we can cover $B(0,R)$ with $(2J+1)^2$ unit discs centered at $(j_1,j_2)\in \sqrt{2}\Z^2$, where $|j_k|\le J\sqrt{2}$ for $k=1,2$. Since $(2J+1)^2\le 12R^2$ for all $R\geq 1$, we obtain \eqref{eq:upper bound for L2 norm}.
Therefore, by \eqref{eq:lower bound for L2 norm}, \eqref{eq:upper bound for L2 norm} and \eqref{eq:appendix},
    \beq\label{eq:app}
    \begin{aligned}
    \veps^2_n(t)-12R^2\veps^2_n(0) & \leq \int_{B(0,R)} |u_n(t,x)|^2 dx-\int_{B(0,R)} |u_n(0,x)|^2 dx \\
&\leq 2 \int_0 ^t\int_{\partial B(0,R)}\left| u_n(\tau,x)\right| \left|\nabla u_n(\tau,x)\right| ds d\tau.
    \end{aligned}
    \eeq

Now we assume that $\veps_n(t)> \sqrt{6}\,\veps_n(0)$, otherwise we clearly have \eqref{eq:claim-app}.
Choose $R_n>1$ such that $\veps_n(t)-\sqrt{6}\,R_n \veps_n(0)=0$, then we have
\[
\int_1^{R_n} \veps_n^2(t)-12R^2\veps_n^2(0) \,dR = \frac{\veps_n^3(t)}{3\sqrt{6} \,\veps_n(0)}  -\veps_n^2(t)+4\veps_n^2(0).
\]
Hence, it follows from \eqref{eq:app} that
\[
\begin{aligned}
\frac{\veps_n^3(t)}{3\sqrt{6} \,\veps_n(0) } -\veps_n^2(t)+4\veps_n^2(0) &\leq 2\int_0^t\int_{\R^2}  |u_n(\tau,x)| |\nabla u_n(\tau,x)|dx d\tau \\
&\leq 2 \int_0^t \|u_n(\tau)\|_{L^2}\|\nabla u_n(\tau)\|_{L^2} d\tau \\
&= 2  \int_0^t \|f_n\|_{L^2}\|\nabla f_n\|_{L^2} d\tau \leq 2C  t
\end{aligned}
\]
as $\|f_n\|_{L^2}=1$ and $\|f_n\|_{H^1}\le C$ for all $n$.
Thus, we have
\[
\left( \veps_n(t) - \sqrt{6} \veps_n(0)\right)^3 \leq 18\veps_n ^2(0)\veps_n(t) + 6\sqrt{6}\,C \veps_n(0) \,t.
\]
Since  $\veps_n(t)\leq 1$, this proves \eqref{eq:claim-app}.
\end{proof}


\noindent
\textbf{Acknowledgements:} The authors are supported by the National Research Foundation of Korea (NRF) grants funded by the Korean government (MSIT) NRF-2020R1A2C1A01010735 and RS-2023-00208824.
\renewcommand{\thesection}{\arabic{chapter}.\arabic{section}}
\renewcommand{\theequation}{\arabic{chapter}.\arabic{section}.\arabic{equation}}
\renewcommand{\thetheorem}{\arabic{chapter}.\arabic{section}.\arabic{theorem}}

 \bibliography{critical-bibfile}

\begin{thebibliography}{10}

\bibitem{Abdullaev}
F.~K. Abdullaev, B.~B. Baizakov, and M.~Salerno.
\newblock Stable two-dimensional dispersion-managed soliton.
\newblock {\em Phys. Rev. E}, 68:066605, 2003.

\bibitem{Brezis}
H.~Br\'{e}zis and E.~Lieb.
\newblock A relation between pointwise convergence of functions and convergence
  of functionals.
\newblock {\em Proc. Amer. Math. Soc.}, 88(3):486--490, 1983.

\bibitem{choi2023global}
M.-R. Choi, Y.~Hong, and Y.-R. Lee.
\newblock Global existence versus finite time blowup dichotomy for the
  dispersion managed {NLS}.
\newblock {\em arXiv}, 2024.

\bibitem{CHLT}
M.-R. Choi, D.~Hundertmark, and Y.-R. Lee.
\newblock Thresholds for existence of dispersion management solitons for
  general nonlinearities.
\newblock {\em SIAM J. Math. Anal.}, 49(2):1519--1569, 2017.

\bibitem{CHLW}
M.-R. Choi, D.~Hundertmark, and Y.-R. Lee.
\newblock Well–posedness of dispersion managed nonlinear {S}chr\"odinger
  equations.
\newblock {\em J. Math. Anal. Appl.}, 522(1):126938, 2023.

\bibitem{CKL}
M.-R. Choi, Y.~Kang, and Y.-R. Lee.
\newblock On dispersion managed nonlinear {S}chr\"{o}dinger equations with
  lumped amplification.
\newblock {\em J. Math. Phys.}, 62(7):071506, 16, 2021.

\bibitem{CLA}
M.-R. Choi and Y.-R. Lee.
\newblock Averaging of dispersion managed nonlinear {S}chr\"{o}dinger
  equations.
\newblock {\em Nonlinearity}, 35(4):2121--2133, 2022.

\bibitem{Chraplyvy}
A.~Chraplyvy, A.~Gnauck, R.~Tkach, and R.~Derosier.
\newblock 8x10 gb/s transmission through 280 km of dispersion-managed fiber.
\newblock {\em IEEE Photonics Tech. Lett.}, 5(10):1233--1235, 1993.

\bibitem{EHL}
M.~B. Erdo\u{g}an, D.~Hundertmark, and Y.-R. Lee.
\newblock Exponential decay of dispersion managed solitons for vanishing
  average dispersion.
\newblock {\em Math. Res. Lett.}, 18(1):11--24, 2011.

\bibitem{Foschi2007}
D.~Foschi.
\newblock Maximizers for the {S}trichartz inequality.
\newblock {\em J. Eur. Math. Soc.}, 009(4):739--774, 2007.

\bibitem{Gabitov96}
I.~Gabitov and S.~Turitsyn.
\newblock Breathing solitons in optical fiber links.
\newblock {\em JETP Lett.}, 63:861--866, 1996.

\bibitem{Gabitov96b}
I.~R. Gabitov and S.~K. Turitsyn.
\newblock Averaged pulse dynamics in a cascaded transmission system with
  passive dispersion compensation.
\newblock {\em Opt. Lett.}, 21(5):327--329, 1996.

\bibitem{GH}
W.~R. Green and D.~Hundertmark.
\newblock Exponential decay of dispersion-managed solitons for general
  dispersion profiles.
\newblock {\em Lett. Math. Phys.}, 106(2):221--249, 2016.

\bibitem{Guevara}
C.~D. Guevara.
\newblock {Global behavior of finite energy solutions to the d-dimensional
  focusing nonlinear {S}chr\"odinger equation}.
\newblock {\em Applied Mathematics Research eXpress}, 2014(2):177--243, 2013.

\bibitem{HL2009}
D.~Hundertmark and Y.-R. Lee.
\newblock Decay estimates and smoothness for solutions of the dispersion
  managed non-linear {S}chr\"{o}dinger equation.
\newblock {\em Comm. Math. Phys.}, 286(3):851--873, 2009.

\bibitem{HL}
D.~Hundertmark and Y.-R. Lee.
\newblock On non-local variational problems with lack of compactness related to
  non-linear optics.
\newblock {\em J. Nonlinear Sci.}, 22(1):1--38, 2012.

\bibitem{HLRZ}
D.~Hundertmark, Y.-R. Lee, T.~Ried, and V.~Zharnitsky.
\newblock Solitary waves in nonlocal {NLS} with dispersion averaged saturated
  nonlinearities.
\newblock {\em J. Differential Equations}, 265(8):3311--3338, 2018.

\bibitem{8180258}
D.~Hundertmark and V.~Zharnitsky.
\newblock On sharp {S}trichartz inequalities in low dimensions.
\newblock {\em Int. Math. Res. Not.}, 2006(9):34080--34080, 2006.

\bibitem{Ikoma}
N.~Ikoma.
\newblock Compactness of minimizing sequences in nonlinear {S}chr\"odinger
  systems under multiconstraint conditions.
\newblock {\em Adv. Nonlinear Stud.}, 14(1):115--136, 2014.

\bibitem{jeanjean}
L.~Jeanjean and S.-S. Lu.
\newblock On global minimizers for a mass constrained problem.
\newblock {\em Calc. Var. Partial Differential Equations}, 61(214):1--18, 2022.

\bibitem{Kunze}
M.~Kunze.
\newblock On a variational problem with lack of compactness related to the
  {S}trichartz inequality.
\newblock {\em Calc. Var. Partial Differential Equations}, 19(3):307--336,
  2004.

\bibitem{Lions1}
P.-L. Lions.
\newblock The concentration-compactness principle in the calculus of
  variations. {The} locally compact case, part 1.
\newblock {\em Ann. Inst. H. Poincar\'{e} Anal. Non Lin\'{e}aire},
  1(2):109--145, 1984.

\bibitem{Lions2}
P.-L. Lions.
\newblock The concentration-compactness principle in the calculus of
  variations. {T}he locally compact case. part 2.
\newblock {\em Ann. Inst. H. Poincar\'{e} Anal. Non Lin\'{e}aire},
  1(4):223--283, 1984.

\bibitem{PEREZGARCIA200631}
V.~M. Pérez-García, P.~J. Torres, and V.~V. Konotop.
\newblock Similarity transformations for nonlinear {S}chr\"odinger equations
  with time-dependent coefficients.
\newblock {\em Phys. D}, 221(1):31--36, 2006.

\bibitem{Serkin}
V.~N. Serkin and A.~Hasegawa.
\newblock Novel soliton solutions of the nonlinear {S}chr\"odinger equation
  model.
\newblock {\em Phys. Rev. Lett.}, 85:4502--4505, Nov 2000.

\bibitem{Stanislavova}
M.~Stanislavova.
\newblock Regularity of ground state solutions of dispersion managed nonlinear
  {S}chr\"{o}dinger equations.
\newblock {\em J. Differential Equations}, 210(1):87--105, 2005.

\bibitem{SS}
C.~Sulem and P.-L. Sulem.
\newblock {\em The nonlinear {S}chr\"{o}dinger equation: Self-focusing and wave
  collapse}, volume 139 of {\em Applied Mathematical Sciences}.
\newblock Springer-Verlag, New York, 1999.

\bibitem{ZGJT01}
V.~Zharnitsky, E.~Grenier, C.~K. R.~T. Jones, and S.~K. Turitsyn.
\newblock Stabilizing effects of dispersion management.
\newblock {\em Phys. D}, 152/153:794--817, 2001.

\end{thebibliography}

\bibliographystyle{abbrv}

\def\cprime{$'$}
\small{
%

}

\end{document}